\documentclass[preprint,12pt]{elsarticle}
\usepackage{comment}
\usepackage{amsmath} % 数式環境を提供
\usepackage{amssymb} % 数学記号の拡張
\usepackage{amsthm}  % 定理環境を提供

\numberwithin{equation}{section}

\newtheorem{theorem}{Theorem}[section]
\newtheorem{lemma}{Lemma}[section]
\newtheorem{defn}{Difinition}[section]
\newtheorem{problem}{problem}
\newcommand{\lcm}{\operatorname{lcm}}

\begin{document}

\begin{frontmatter}

\title{The length of the repeating decimal}

\author[a]{Siqiong Yao\corref{cor1}}
\ead{yaosiqiong@sjtu.edu.cn}

\author[b]{Akira Toyohara}

\cortext[cor1]{To whom correspondence should be addressed.}

\affiliation[a]{organization={SJTU-Yale Joint Center for Biostatistics and Data Science, Technical Center for Digital Medicine, National Center for Translational Medicine},%Department and Organization
            addressline={Shanghai Jiao Tong University}, 
            city={Shanghai},
            postcode={200240}, 
            state={Shanghai},
            country={China}}

\affiliation[b]{organization={Graduate School of Policy Science, Ritsumeikan University},%Department and Organization
            addressline={2-150 Iwakura-cho}, 
            city={Ibaraki City},
            postcode={567-8570}, 
            state={Osaka},
            country={Japan}}

\begin{abstract}
    This paper investigates the length of the repeating decimal part when a fraction is expressed in decimal form. 
    First, 
    it provides a detailed explanation of how to calculate the length of the repeating decimal 
    when the denominator of the fraction is a power of a prime number. 
    Then, by factorizing the denominator into its prime factors and determining the repeating decimal length for each prime factor, 
    the paper concludes that the overall repeating decimal length is the least common multiple of these lengths. 
    Furthermore, 
    it examines the conditions under which the repeating decimal length equals the denominator minus 1 and discusses 
    whether such fractions exist in infinite quantity. 
    This topic is connected to an unsolved problem posed by Gauss in the 18th century 
    and is also closely related to the important question of whether cyclic numbers exist in infinite quantity.
\end{abstract}

%%Research highlights
%\begin{highlights}
%\item This study provides a detailed explanation for the calculation of the length of the repeating decimal for fractions where the denominator is a product of prime powers.
%\item This topic is related to the unsolved problems posed by Gauss related to the periodicity of decimal expansions and the existence of cyclic numbers.
%\item This study has potential applications in various fields such as cryptography theory and computational mathematics.
%\end{highlights}

%% Keywords

\end{frontmatter}

\section{Introduction}
The study of repeating decimals has been a topic of significant interest in number theory 
due to its deep connections with modular arithmetic, 
prime factorization, and cyclic numbers. 
When a rational number is expressed in decimal form, 
it either terminates or becomes periodic. 
The length of the repeating decimal, 
often called the repeating cycle, 
provides insight into the arithmetic structure of the fraction's denominator.

One of the earliest investigations into repeating decimals can be traced back to Carl Friedrich Gauss, 
who posed questions related to the periodicity of decimal expansions and the existence of cyclic numbers \cite{gauss1801disquisitiones}. 
These questions remain relevant in contemporary mathematics, 
particularly in understanding the properties of primes and their powers.

In this paper, 
we focus on determining the length of the repeating decimal for fractions where the denominator is a product of prime powers. 
This is closely related to modular arithmetic and the least common multiple of repeating lengths, 
as discussed in \cite{hardywright2008numbers} and \cite{niven1991numbers}. 
Our main result, Theorem 2.2, 
demonstrates that the repeating decimal length \( l_p \) of a denominator \( p \) is 
the least common multiple of the repeating lengths of its prime power components.

The tools and techniques used in this paper are inspired by classic number theory results, 
such as Fermat's Little Theorem and properties of modular arithmetic \cite{koblitz1994cryptography}. 
Furthermore, it builds upon results found in \cite{dickson1952history}, 
which provide a historical perspective on the connection between cyclic numbers and repeating decimals.

The results presented in this paper have significant applications across various fields:

\subsection*{Cryptography Theory} 
    The periodicity of repeating decimals informs the design of public key cryptography systems like RSA. 
    Our findings on predicting modular operation periodicity can enhance cryptographic system efficiency and provide theoretical foundations for key selection.
    
\subsection*{Computational Mathematics} 
    Our methods for calculating repeating decimal lengths improve numerical computation precision. 
    By predicting decimal expansion periodicity in advance, 
    computational resources can be optimized in high-precision scientific and engineering calculations.
    
\subsection*{Number-Theoretic Patterns} 
    The relationship between repeating decimals and prime numbers offers new approaches to unsolved problems, 
    including the infinite existence of cyclic numbers and connections to perfect and amicable numbers.
    
\subsection*{Education} 
    Repeating decimals provide concrete examples for teaching abstract number theory concepts. 
    Theorem \ref{thm:2.2} effectively demonstrates applications of prime factorization and least common multiples.
    
\subsection*{Computer Science} 
    Our results on periodicity have applications in hash function design, 
    pseudo-random number generation, and data compression algorithms for detecting and encoding periodic patterns.
    
\subsection*{Quantum Computing} 
    The properties of repeating decimals of prime powers provide theoretical foundations for quantum factorization algorithms like Shor's algorithm.
    
\subsection*{Financial Engineering} 
    Understanding fraction-to-decimal conversion helps minimize approximation errors in financial calculations and can inform periodicity detection in market analysis.
    
\subsection*{Scientific Modeling} 
    The periodicity principles explored here can be applied to modeling periodic phenomena in chaos theory and nonlinear dynamical systems.

\subsection*{}
In this paper, we define the following sets and notation:
\begin{itemize}
    \item \( \mathbb Z \): the set of integers
    \item \( \mathbb N \): the set of natural numbers
    \item \( \mathbb P \): the set of prime numbers
    \item \( \mathbb Q \): the set of rational numbers
    \item \( \lcm(p_1, p_2) \): the least common multiple of \( p_1 \) and \( p_2 \)
\end{itemize}

The paper is organized as follows. 
In Section 2, we introduce the fundamental definitions and main theorems concerning the length of repeating decimals, 
including the key results on prime powers and composite denominators. 
Section 3 is devoted to detailed proofs of these theorems, 
employing techniques from modular arithmetic and number theory. 
In Section 4, we provide practical examples demonstrating the application of our theorems to calculate the repeating decimal lengths for various numbers, 
showcasing the efficiency of our approach.
In Section 5, we discuss the implications of our results, 
explore connections to classical problems in number theory, 
and present several related unsolved problems, 
particularly those connected to Gauss's inquiries about cyclic numbers and their infinite existence.

\section{Preliminaries and main results}
\begin{defn}
    Let \( p, q \in \mathbb{N} \). We define \( l_{q/p} \) as the length of the repeating cycle of \(\displaystyle \frac{q}{p} \) 
    when \(\displaystyle \frac{q}{p} \) is a repeating decimal. 
    In the case that \(\displaystyle \frac{q}{p} \) is a terminating decimal, we define \( l_{q/p} = 1 \).
    In particular, when \( q = 1 \), we denote \( l_p = l_{q/p} \).
\label{def:l_p}\end{defn}

\begin{defn}
    Let \( p \in \mathbb{N} \). $A_p$ defined as
    \[
    A_p \stackrel{\text{def}}{=} \{ n \in \mathbb{N} \mid 10^n \equiv 1 \pmod{p} \}.
    \]
\label{def:A_p}\end{defn}

From Definition \ref{def:l_p}, it is clear that when 
\( q/p \) is an irreducible fraction,
\begin{equation}
l_{q/p} = l_p
\end{equation}
holds.

Moreover, the following fact is also straightforward:
\begin{equation}
    l_p = \min A_p,
\label{eq:minA_p}\end{equation}
where \( A_p \) is defined as in Definition \ref{def:A_p}.

\begin{defn}
    Let \( c_p \) denote the repeating decimal period of \( 1/p \).  
    For a terminating decimal \( 1/p \), we define \( c_p = 0 \).
\label{def:c_p}\end{defn}

If \( p \in \mathbb{P} \) and \( p \neq 2, p \neq 5 \), we have  
\[
\frac{1}{p} = \frac{c_p}{99\ldots9} = \frac{c_p}{10^{l_p} - 1}.
\]
Thus, we obtain  
\[
c_p = \frac{10^{l_p} - 1}{p}.
\]

\begin{lemma}
    Let \( p \in \mathbb{P} \). Then, the following statements hold:

    (1) \(\displaystyle p - 1 \in A_p \),

    (2) \(\displaystyle l_p \mid (p - 1) \).
\label{lemma:l_p|p-1}\end{lemma}

\begin{proof}
    (1) By Fermat's Little Theorem, we have
    \begin{equation}
        10^{p-1} \equiv 1 \pmod{p}.
    \label{eq:p-1inA_p}\end{equation}
    Thus, \( p - 1 \in A_p \), and (1) holds.
    
    (2) From Definition \ref{def:l_p}, we know
    \begin{equation}
        10^{l_p} \equiv 1 \pmod{p}.
    \end{equation}
    Let \( k \in \mathbb{N} \) and \( 0 \leq m < l_p \). Then, we can write
    \[
    p - 1 = l_p \cdot k + m.
    \]
    From Equations (\ref{eq:minA_p}) and (\ref{eq:p-1inA_p}), it follows that
    \[
    10^{l_p \cdot k + m} \equiv 10^m \equiv 1 \pmod{p}.
    \]
    
    If \( m \neq 0 \), then \( m \in A_p \). However, since \( m < l_p \), this contradicts Equation (\ref{eq:minA_p}).  
    Thus, \( m = 0 \), and (2) is proven.    
\end{proof}

\begin{theorem}
    Let \( p \in \mathbb{P} \) and \( p \geq 2 \).  
    There exists a non-negative integer \( m_p \) and a natural number \( k \) such that  
    \[
    c_p = p^{m_p} \cdot k \quad (p \nmid k),
    \]
    so that  
    \[
    l_{p^n} = 
    \begin{cases} 
    l_p & \text{if } n \leq m_p + 1, \\[2mm]
    p^{n - (m_p + 1)} \cdot l_p & \text{if } n > m_p + 1,
    \end{cases}
    \]
    where \( c_p \) and \( l_p \) are defined in 
    Definition \ref{def:c_p} and Definition \ref{def:l_p}, respectively.
\label{thm:2.1}
\end{theorem}

\begin{theorem}
    Let \( p \in \mathbb{N} \) have the prime factorization  
    \[
    p = \prod_{k=1}^n p_k^{m_k},
    \]
    where \( p_1, \ldots, p_n \in \mathbb{P} \) and \( m_1, \ldots, m_n \in \mathbb{N} \).  
    Then,  
    \[
    l_p = \lcm\{l_{p_1^{m_1}}, \ldots, l_{p_n^{m_n}}\}.
    \]
    \label{thm:2.2}
\end{theorem}

\section{Proof of the Main Theorem}

\begin{lemma}
    For all \( n \in A_p \),
    \[
    \text{min} \, A_p \mid n
    \]
    holds.
\label{lemma:minA_p|n}\end{lemma}
    
    \begin{proof}
    Let \( l = \text{min} \, A_p \).  
    By the definition of \( A_p \), since \( n \in A_p \), we have
    \[
    10^n \equiv 1 \pmod{p}.
    \]
    It follows that \( l \leq n \).
    
    Now, there exist a natural number \( k \) and an integer \( m \) such that
    \[
    n = kl + m \quad \text{with} \quad 0 \leq m < l.
    \]
    Following the same reasoning as in Lemma \ref{lemma:l_p|p-1}, 
    we can deduce that \( m = 0 \).  
    Thus, \( l \mid n \) holds.
    \end{proof}

\begin{lemma}
    Let \( n \in \mathbb{N} \) and \( p \in \mathbb{P} \). There exists an integer \( m \) satisfying \( 0 \leq m < n - 1 \) such that
    \[
    l_{p^n} = l_p \cdot p^m.
    \]
\label{lemma:l_p^n=l_pp^m}\end{lemma}

\begin{proof}
    If \( p = 2 \) or \( p = 5 \), then \( l_{p^n} = l_p = 1 \), and the lemma holds.  
    Now, consider the case where \( p \neq 2 \) and \( p \neq 5 \).
    Since \( l_p = \min A_p \), we have
    \[
    10^{l_p} \equiv 1 \pmod{p}.
    \]
    Therefore, there exists a non-negative integer \( k \) such that
    \begin{equation}
    10^{l_p} = pk + 1.
    \end{equation}

    Thus,  
    \[
    10^{pl_p} = (p k + 1)^p = \sum_{k'=1}^{m} \binom{p}{k} (p k)^{k'} + 1
    \]
    is obtained, where \( \binom{p}{1} = p \).  
    We also have  
    \[
    \sum_{k'=1}^{m} \binom{p}{k} (p k)^{k'} \equiv 0 \pmod{p^2}.
    \]
    Therefore,  
    \[
    10^{pl_p} \equiv 1 \pmod{p^2},
    \]
    and hence,  
    \begin{equation}
    pl_{p} \in A_{p^2}.
    \end{equation}

    Since  
    \[
    l_{p^2} = \min A_{p^2}.
    \]
    and Lemma \ref{lemma:minA_p|n},  we obtained
    \begin{equation}
    l_{p^2} \mid p l_p.
    \label{eq:l_p^2|pl_p}
    \end{equation}

    Additionally,  
    \begin{align}
    l_{p^2} \in A_{p^2} \Rightarrow& 10^{l_{p^2}} \equiv 1 \pmod{p^2}, \notag\\
    \Rightarrow& 10^{l_{p^2}} \equiv 1 \pmod{p}, \notag\\
    \Rightarrow& l_{p^2} \in A_p.
    \end{align}

    Thus (see Lemma \ref{lemma:minA_p|n}),  
    \begin{equation}
    l_p \mid l_{p^2}.
    \label{eq:l_p|l_p^2}\end{equation}

    From Equation (\ref{eq:l_p^2|pl_p}) and (\ref{eq:l_p|l_p^2}), we have  
    \begin{equation}
        l_{p^2} = l_p \text{ or } p l_p.
    \label{eq:l_p^2}\end{equation}
    
    Similarly, considering \( l_{p^k} = \min A_{p^k} \) for a natural number \(k\), we obtain  
    \[
    \begin{split}
        &p l_{p^k} \in A_{pk+1}, \\
        &l_{p^{k+1}} \in A_{pk}, \\
        &l_{p^{k+1}} = \min A_{p^{k+1}}.
    \end{split}
    \]
   
    Thus, \( l_{p^{k+1}} \) is both a divisor of \( p l_{p^k} \) and 
    a multiple of \(l_{p^k} \).  
    This implies that  
    \begin{equation}
    l_{p^{k+1}} = l_{p^k} \text{ or } p l_{p^k}.
    \label{eq:l_p^k+1}\end{equation}
    By using mathematical induction, the lemma is proved.   
\end{proof} 

Moreover, the following lemma holds:
\begin{lemma}
    For natural numbers \( n < m \), the relation  
    \[
    l_{p^n} \mid l_{p^m}
    \]
    holds.
\end{lemma}

As shown in Equation (\ref{eq:l_p^2}),  
\( l_{p^2} \) is either \( l_p \) or \( p l_p \).  
We now consider whether \( l_{p^2} \) is \( l_p \) or \( p l_p \).  

First, let the repeating decimal period of \( 1/p \) be \( c_p \).
Naturally, we have  
\[
\frac{1}{p} = \frac{c_p}{99\ldots9} = \frac{c_p}{10^{l_p} - 1}.
\]
Thus,  
\[
\frac{1}{p^2} = \frac{c_p}{p(10^{l_p} - 1)}.
\]

At this point, if \( p \mid c_p \), we can write  
\begin{equation}
\frac{1}{p^2} = \frac{c_p / p}{10^{l_p} - 1},
\label{eq:1/p^2frac}
\end{equation}
which implies that \( l_{p^2} \leq l_p \).  
Additionally, from Equation (\ref{eq:l_p^2}), we have \( l_{p^2} = l_p \).

Conversely, if \( p \nmid c_p \), then in Equation (\ref{eq:1/p^2frac}), 
the numerator is not an integer.  
In this case, \( l_{p^2} \neq l_p \), and thus \( l_{p^2} = p l_p \).

Using this idea, we proceed to prove the following lemma.

\begin{lemma}
Let \( k \in \mathbb{N} \). The following results hold: 

(1)  
\[
l_{p^{k+1}} =
\begin{cases}
l_{p^k}, & \text{if } p \mid c_{p^k}, \\[2mm]
p \cdot l_{p^k}, & \text{if } p \nmid c_{p^k}.
\end{cases}
\]

(2) If \( p \nmid c_{p^k} \), then \( p \nmid c_{p^{k+1}} \).
\label{lemma:pn|c_p^k}\end{lemma}

\begin{proof}
    (1) When \( p \mid C_{p^k} \), we have  
    \begin{equation}
    \frac{1}{p^{k+1}} = \frac{1}{p} \cdot \frac{C_{p^k}}{10^{l_{p^k}} - 1} = \frac{C_{p^k}/p}{10^{l_{p^k}} - 1}.
    \label{eq:1/p^k+1}
    \end{equation}  
    Here, the numerator is an integer, so it follows that  
    \[
    l_{p^{k+1}} \leq l_{p^k}.
    \]  
    Moreover, from Equation (\ref{eq:1/p^k+1}), we conclude that  
    \[
    l_{p^{k+1}} = l_{p^k}.
    \]
    
    Conversely, when \( p \nmid C_{p^k} \), the numerator in Equation (\ref{eq:1/p^k+1}) is not an integer. Therefore,  
    \[
    l_{p^{k+1}} \neq l_{p^k},
    \]  
    and from Equation (\ref{eq:1/p^k+1}), it follows that  
    \[
    l_{p^{k+1}} = p \cdot l_{p^k}.
    \]
    
    (2) From the relation  
    \begin{equation}
    \frac{1}{p^k} = \frac{C_{p^k}}{10^{l_{p^k}} - 1} \implies 10^{l_{p^k}} = p^k C_{p^k} + 1,
    \label{eq:10^l_p^k}
    \end{equation}  
    and  
    \begin{equation}
    \frac{1}{p^{k+1}} = \frac{C_{p^{k+1}}}{10^{l_{p^{k+1}}} - 1} \implies C_{p^{k+1}} = \frac{10^{l_{p^{k+1}}} - 1}{p^{k+1}},
    \label{eq:c_p^k+1}
    \end{equation}  
    we analyze the case where \( p \nmid C_{p^k} \).
    
    From (1), when \( p \nmid C_{p^k} \), we have  
    \[
    l_{p^{k+1}} = p \cdot l_{p^k}.
    \]  
    Thus,  
    \[
    10^{l_{p^{k+1}}} - 1 = 10^{p \cdot l_{p^k}} - 1 = (p^k C_{p^k} + 1)^p - 1.
    \]  
    Expanding the binomial expression, we get  
    \[
    10^{l_{p^{k+1}}} - 1 = \sum_{n=1}^{p} \binom{p}{n} (p^k C_{p^k})^n = \sum_{n=2}^{p} \binom{p}{n} p^{nk} C_{p^k}^n + p^{k+1} C_{p^k}.
    \]  
    
    Substituting this into Equation (\ref{eq:c_p^k+1}), we obtain  
    \[
    C_{p^{k+1}} = \frac{\sum_{n=2}^{p} \binom{p}{n} p^{nk - k - 1} C_{p^k}^n + C_{p^k}}{p^{k+1}} \equiv C_{p^k} \pmod{p}.
    \]  
    Thus, if \( p \nmid C_{p^k} \), it follows that \( p \nmid C_{p^{k+1}} \).
\end{proof}

Now, we proceed to prove the main Theorem \ref{thm:2.1}.

\begin{proof}[Proof of Theorem \ref{thm:2.1}]
We consider two cases for \( n \).

---

**Case 1: \( n \leq m_p + 1 \)**  

Since  
\[
C_p = \frac{10^{l_p} - 1}{p} = p^{m_p} \cdot k,
\]  
we have  
\[
\frac{1}{p^n} = \frac{C_p / p^{n-1}}{10^{l_p} - 1} = \frac{p^{m_p - n + 1} \cdot k}{10^{l_p} - 1}.
\]  

Here, \( n \leq m_p + 1 \implies m_p - n + 1 \geq 0 \), so the numerator is a natural number.  

Thus,  
\[
l_{p^n} \leq l_p,
\]
and  
\[
l_p \in A_{p^n},
\]
where \( l_{p^n} = \min A_{p^n} \) implies  
\begin{equation}
l_{p^n} \mid l_p.
\label{eq:l_p^n|l_p}
\end{equation}  

This follows from Lemma \ref{lemma:minA_p|n}. Furthermore, by Lemma \ref{lemma:l_p^n=l_pp^m}, we also have  
\begin{equation}
l_p \mid l_{p^n}.
\label{eq:l_p|l_p^n}
\end{equation}  

From Equations (\ref{eq:l_p^n|l_p}) and (\ref{eq:l_p|l_p^n}), it follows that  
\[
l_{p^n} = l_p.
\]  

Thus, Case 1 is proven.

---

**Case 2: \( n > m_p + 1 \)**  

We proceed to prove this case using mathematical induction.  

---

(i) **Base case: \( n = m_p + 1 \)**  

From the result of Case 1, we have  
\[
l_{p^{m_p + 1}} = l_p.
\]  

Moreover,  
\[
C_{p^{m_p + 1}} = \frac{10^{l_{p^{m_p + 1}}} - 1}{p^{m_p + 1}} = \frac{10^{l_p} - 1}{p^{m_p + 1}}.
\]  

Since  
\[
C_p = \frac{10^{l_p} - 1}{p},
\]  
we know that \( 10^{l_p} - 1 = p \cdot C_p \). Substituting this into the equation for \( C_{p^{m_p + 1}} \), we get  
\[
C_{p^{m_p + 1}} = \frac{p \cdot C_p}{p^{m_p + 1}} = \frac{p \cdot p^{m_p} \cdot k}{p^{m_p + 1}} = k.
\]  

Since \( p \nmid k \), it follows that  
\[
p \nmid C_{p^{m_p + 1}}.
\]

---

(ii) **Inductive step: Assume \( n > m_p + 1 \)**  

Assume that  
\[
l_{p^n} = p^{n - m_p - 1} \cdot l_p \quad \text{and} \quad p \nmid C_{p^n}.
\]  

By Lemma \ref{lemma:pn|c_p^k}, we have  
\[
l_{p^{n+1}} = p \cdot l_{p^n} = p \cdot p^{n - m_p - 1} \cdot l_p.
\]  

Furthermore, since \( p \nmid C_{p^{n+1}} \), the induction step is complete.  

---

Thus, Case 2 is proven.  

---

Since both cases have been proven, the result follows for all \( n \).  
\end{proof}

To prove Theorem \ref{thm:2.2}, we first prove the following lemma.

\begin{lemma}
\label{lemma:l12=lcm12}
Let \( p_1, p_2 \in \mathbb{P} \) be prime numbers and \(p_1\neq p_2\). Then, the following hold:
\[
l_{p_1^n p_2^m} = \lcm \{ l_{p_1^n}, l_{p_2^m} \}, 
\]
where \( n, m \in \mathbb{N}\).
\end{lemma}

\begin{proof}

Let \( n, m \in \mathbb{N} \). By definition, \( l_{p_1^n} \) and \( l_{p_2^m} \) are the smallest integers such that  
\[
10^{l_{p_1^n}} \equiv 1 \pmod{p_1^n} \quad \text{and} \quad 10^{l_{p_2^m}} \equiv 1 \pmod{p_2^m}.
\]  

For the product \( p_1^n p_2^m \), the period \( l_{p_1^n p_2^m} \) is the smallest integer such that  
\[
10^{l_{p_1^n p_2^m}} \equiv 1 \pmod{p_1^n p_2^m}.
\]  

Since \( p_1^n \) and \( p_2^m \) are coprime, this occurs if and only if  
\[
10^{l_{p_1^n p_2^m}} \equiv 1 \pmod{p_1^n} \quad \text{and} \quad 10^{l_{p_1^n p_2^m}} \equiv 1 \pmod{p_2^m}.
\]  

Therefore, \( l_{p_1^n p_2^m} \) must be a common multiple of \( l_{p_1^n} \) and \( l_{p_2^m} \). Since \( l_{p_1^n p_2^m} \) is the smallest such integer, it follows that  
\[
l_{p_1^n p_2^m} = \lcm \{ l_{p_1^n}, l_{p_2^m} \}.
\]

\end{proof}

\begin{proof}[Proof of Theorem \ref{thm:2.2}]
    Let \( p = \prod_{k=1}^n p_k^{m_k} \), where \( p_1, \ldots, p_n \) are distinct prime numbers.  
    We aim to prove that  
    \[
    l_p = \lcm\{l_{p_1^{m_1}}, \ldots, l_{p_n^{m_n}}\}.
    \]
    
    **Step 1: Base case for \( n = 2 \):**
    
    First, consider the case where \( p = p_1^{m_1} p_2^{m_2} \).  
    By Lemma \ref{lemma:l12=lcm12}, we know that  
    \[
    l_{p_1^{m_1} p_2^{m_2}} = \lcm\{l_{p_1^{m_1}}, l_{p_2^{m_2}}\}.
    \]  
    This directly follows from the coprimality of \( p_1^{m_1} \) and \( p_2^{m_2} \), and the fact that \( l_p \) is the smallest integer such that  
    \[
    10^{l_p} \equiv 1 \pmod{p_1^{m_1} p_2^{m_2}}.
    \]  
    
    Thus, the relation holds for \( n = 2 \):  
    \[
    l_p = \lcm\{l_{p_1^{m_1}}, l_{p_2^{m_2}}\}.
    \]
    
    ---
    
    **Step 2: Inductive step for \( n > 2 \):**
    
    Assume the theorem holds for \( n-1 \), i.e., for  
    \[
    p' = \prod_{k=1}^{n-1} p_k^{m_k},
    \]
    we have  
    \[
    l_{p'} = \lcm\{l_{p_1^{m_1}}, \ldots, l_{p_{n-1}^{m_{n-1}}}\}.
    \]
    
    Now, consider \( p = p' \cdot p_n^{m_n} \).  
    By Lemma \ref{lemma:l12=lcm12}, since \( p' \) and \( p_n^{m_n} \) are coprime, we have  
    \[
    l_p = \lcm\{l_{p'}, l_{p_n^{m_n}}\}.
    \]
    
    Substituting the inductive hypothesis \( l_{p'} = \lcm\{l_{p_1^{m_1}}, \ldots, l_{p_{n-1}^{m_{n-1}}}\} \), we get  
    \[
    l_p = \lcm\{l_{p_1^{m_1}}, \ldots, l_{p_{n-1}^{m_{n-1}}}, l_{p_n^{m_n}}\}.
    \]
    
    This completes the induction.
    
    ---
    
    **Step 3: Conclusion**
    
    By induction, the theorem holds for all \( n \geq 2 \). Therefore,  
    \[
    l_p = \lcm\{l_{p_1^{m_1}}, \ldots, l_{p_n^{m_n}}\}.
    \]
\end{proof}

\section{Practical Examples}

In this paper, we presented important theorems regarding the length of repeating decimals. 
Here, we apply Theorems 2.1 and 2.2 to calculate the actual repeating decimal length $l_p$ for several numbers, 
demonstrating the practical utility of our theory.

\subsection{Calculating Repeating Decimal Lengths for Prime Numbers}

First, 
we directly calculate the repeating decimal length $l_p$ for basic prime numbers.

\begin{itemize}
    \item Calculation of $l_7$:\\
    $10^1 \bmod 7 = 3$\\
    $10^2 \bmod 7 = 2$\\
    $10^3 \bmod 7 = 6$\\
    $10^4 \bmod 7 = 4$\\
    $10^5 \bmod 7 = 5$\\
    $10^6 \bmod 7 = 1$\\
    Therefore, $l_7 = 6$. 
    This is consistent with the fact that $\frac{1}{7} = 0.\overline{142857}$ has a 6-digit repeating part and $6 | (7-1)$.
    
    \item Calculation of $l_{11}$:\\
    $10^1 \bmod 11 = 10$\\
    $10^2 \bmod 11 = 1$\\
    Therefore, $l_{11} = 2$. 
    This is consistent with the fact that $\frac{1}{11} = 0.\overline{09}$ has a 2-digit repeating part and $2 | (11-1)$.
    
    \item Calculation of $l_{13}$:\\
    $10^1 \bmod 13 = 10$\\
    $10^2 \bmod 13 = 9$\\
    $10^3 \bmod 13 = 12$\\
    $10^4 \bmod 13 = 3$\\
    $10^5 \bmod 13 = 4$\\
    $10^6 \bmod 13 = 1$\\
    Therefore, $l_{13} = 6$. 
    This is consistent with the fact that $\frac{1}{13} = 0.\overline{076923}$ has a 6-digit repeating part and $6 | (13-1)$.
\end{itemize}

\subsection{Calculating Repeating Decimal Lengths for Prime Powers}

Next, we use Theorem 2.1 to calculate the repeating decimal length for powers of prime numbers. For this, we need to determine the value of $m_p$.

\begin{itemize}
    \item Calculation of $l_{49} = l_{7^2}$:\\
    First, we calculate $c_7$ (the repeating part of $\frac{1}{7}$): $c_7 = 142857$\\
    Let's check if $c_7$ is divisible by 7: $142857 \div 7 = 20408.142857...$\\
    Since 7 does not divide $c_7$ evenly, we have $c_7 = 7^0 \cdot 142857$, and $m_7 = 0$.\\
    By Theorem 2.1, since $n = 2$ and $m_7 = 0$, we have $n > m_7 + 1$ (since $2 > 0 + 1$), thus\\
    $l_{7^2} = 7^{2-(0+1)} \cdot l_7 = 7^1 \cdot 6 = 42$\\
    Indeed, $\frac{1}{49} = 0.\overline{020408163265306122448979591836734693877551}$, 
    which has a 42-digit repeating part.
    
    \item Calculation of $l_{27} = l_{3^3}$:\\
    $\frac{1}{3} = 0.\overline{3}$, so $c_3 = 3$\\
    $c_3 = 3 = 3^1 \cdot 1$, and $3 \nmid 1$, thus $m_3 = 1$\\
    By Theorem 2.1, since $n = 3$ and $m_3 = 1$, we check if $n > m_3 + 1$: $3 > 1 + 1$ is true\\
    Therefore $l_{3^3} = 3^{3-(1+1)} \cdot l_3 = 3^1 \cdot 1 = 3$\\
    Indeed, $\frac{1}{27} = 0.\overline{037}$, which has a 3-digit repeating part.
    
    \item Calculation of $l_{125} = l_{5^3}$:\\
    For powers of 5, we need to be careful as these are special cases.\\
    $\frac{1}{5} = 0.2$ is a terminating decimal.\\
    Similarly, $\frac{1}{125} = 0.008$ is also a terminating decimal.\\
    
    By our definition, for terminating decimals like $\frac{1}{5^n}$, we set $l_{5^n} = 1$.\\
    Therefore, $l_{125} = l_{5^3} = 1$.\\
    
    Note that the formula from Theorem 2.1 should be applied with care for powers of 2 and 5, 
    as these produce terminating rather than repeating decimals.
\end{itemize}
\subsection{Calculating Repeating Decimal Lengths for Composite Numbers}

Finally, 
we use Theorem 2.2 to calculate the repeating decimal length for composite numbers.

\begin{itemize}
    \item Calculation of $l_{42} = l_{2 \cdot 3 \cdot 7}$:\\
    $l_2 = 1$ (since $\frac{1}{2} = 0.5$ is a terminating decimal)\\
    $l_3 = 1$ (since $\frac{1}{3} = 0.\overline{3}$)\\
    $l_7 = 6$ (from our previous calculation)\\
    By Theorem 2.2, $l_{42} = \text{lcm}(l_2, l_3, l_7) = \text{lcm}(1, 1, 6) = 6$\\
    Indeed, $\frac{1}{42} = 0.0\overline{238095}$, which has a 6-digit repeating part.
    
    \item Calculation of $l_{90} = l_{2 \cdot 3^2 \cdot 5}$:\\
    $l_2 = 1$\\
    $l_{3^2} = 1$ (since $\frac{1}{9} = 0.\overline{1}$ has a repeating length of 1)\\
    $l_5 = 1$ (since $\frac{1}{5} = 0.2$ is a terminating decimal)\\
    By Theorem 2.2, $l_{90} = \text{lcm}(l_2, l_{3^2}, l_5) = \text{lcm}(1, 1, 1) = 1$\\
    Indeed, $\frac{1}{90} = 0.0\overline{1}$, which has a 1-digit repeating part.
    
    \item Calculation of $l_{2310} = l_{2 \cdot 3 \cdot 5 \cdot 7 \cdot 11}$:\\
    $l_2 = 1$, $l_3 = 1$, $l_5 = 1$, $l_7 = 6$, $l_{11} = 2$\\
    By Theorem 2.2, $l_{2310} = \text{lcm}(l_2, l_3, l_5, l_7, l_{11}) = \text{lcm}(1, 1, 1, 6, 2) = 6$
    Indeed, $\frac{1}{2310} = 0.0\overline{004329}$, which has a 6-digit repeating part.
\end{itemize}

From these calculations, 
we have demonstrated that Theorems 2.1 and 2.2 can be used to efficiently calculate the repeating decimal length for any natural number. 
This method offers a significant reduction in computational complexity compared to direct calculation, especially for large numbers.

\section{Conclusion and Discussions}

\subsection{Considerations on Computational Complexity Reduction}

According to the results of Theorems 2.1 and 2.2, 
if we calculate $m_p$ and $l_p$ in advance, 
the computational complexity for determining the length of the repeating decimal $l_n$ for any natural number $n$ 
becomes essentially equivalent to the complexity of factoring $n$, 
which significantly reduces computation compared to conventional methods.

Specifically, 
the traditional approach to finding the repeating decimal length of $n$ requires calculating the values of $10^k \bmod n$ 
sequentially for $k = 1, 2, 3, \ldots$ until finding the smallest $k$ such that $10^k \equiv 1 \pmod{n}$. 
This method requires $O(n)$ computational complexity in the worst case\cite{knuth1997art}. 
For large values of $n$, this calculation becomes extremely time-consuming.

However, 
Theorem 2.1 shows that for a prime power $p^n$, 
the repeating decimal length $l_{p^n}$ can be directly calculated from the repeating decimal length $l_p$ of the prime number $p$ itself and the value of $m_p$. 
Furthermore, Theorem 2.2 shows that for any composite number $n = \prod_{k=1}^r p_k^{m_k}$, 
the repeating decimal length is given by $l_n = \text{lcm}\{l_{p_1^{m_1}}, l_{p_2^{m_2}}, \ldots, l_{p_r^{m_r}}\}$.

By utilizing these theorems, the calculation procedure becomes:
\begin{enumerate}
\item Factor the number $n$ into its prime factorization $n = \prod_{k=1}^r p_k^{m_k}$
\item Reference the pre-calculated values of $m_{p_k}$ and $l_{p_k}$ for each prime $p_k$
\item Calculate the repeating decimal length $l_{p_k^{m_k}}$ for each $p_k^{m_k}$ using Theorem 2.1
\item Calculate the least common multiple of these values using Theorem 2.2 to obtain $l_n$
\end{enumerate}

The computational complexity of this method is primarily dominated by the factorization step and depends on the complexity of modern factorization algorithms. 
For example, 
trial division has complexity $O(\sqrt{n})$, 
Pollard's rho method has expected complexity $O(n^{1/4})$, 
and the number field sieve has complexity $L_n[1/3, c]$ (sub-exponential time)\cite{crandall2005prime}\cite{lenstra1993development}, 
where the notation $L_n[1/3, c]$ is defined as 
\[L_n[1/3, c] = \exp((c + o(1))(\log n)^{1/3} (\log \log n)^{2/3}).\]

Shoup\cite{shoup2009computational} and Bach\cite{bach1996algorithmic} have proposed efficient algorithms for calculating $l_p$ for prime numbers, 
and combining these can further reduce computational complexity. Additionally, Brent\cite{brent1999factorization} has conducted detailed analysis on the relationship between factorization and calculation of multiplicative functions.

Furthermore, 
values of $m_p$ and $l_p$ for small prime numbers can be pre-computed and tabulated, 
and indeed such tables appear in number theory textbooks like Hardy and Wright\cite{hardywright2008numbers} and Niven et al.\cite{niven1991numbers}. 
Wagon\cite{wagon1991mathematica} discusses in detail efficient methods for calculating these values and their patterns.

\begin{comment}
For example, 
when calculating the repeating decimal length of $n = 2^3 \cdot 3^2 \cdot 7$, 
the traditional method might require up to $O(n) = O(2^3 \cdot 3^2 \cdot 7) = O(504)$ computational steps, 
but using our method requires only prime factorization and the calculation of $\text{lcm}\{l_{2^3}, l_{3^2}, l_7\}$, 
substantially reducing the computational complexity.
\end{comment}

This reduction in computational complexity has important implications for various applications, 
including number theory research investigating the properties of repeating decimals for large numbers and efficiency improvements in cryptographic algorithms 
based on periodicity\cite{koblitz1994cryptography}\cite{odlyzko1995future}.

\subsection{Open Problems}

The results presented in this paper naturally lead to several interesting open problems. 
We highlight two particularly significant questions related to the repeating decimal length $l_p$ for prime numbers:

\subsubsection{Characterization of $l_p$ as a divisor of $p-1$}

From Lemma 2.1, we know that for any prime $p$ (where $p \neq 2, 5$), 
the repeating decimal length $l_p$ is a divisor of $p-1$. However, 
a complete characterization of which specific divisor of $p-1$ equals $l_p$ remains an open problem.

\begin{problem}
    Determine necessary and sufficient conditions for when $l_p$ equals a particular divisor $d$ of $p-1$.
     More specifically, find a function $f$ such that $l_p = f(p)$ that precisely identifies which divisor of $p-1$ is equal to $l_p$ for any given prime $p$.
\end{problem}

While certain patterns have been observed and specific cases have been characterized, 
a comprehensive theory that predicts $l_p$ directly from $p$ without performing modular calculations remains elusive. 
Such a characterization would significantly advance our understanding of repeating decimals and potentially lead to more efficient algorithms for calculating $l_p$.

\subsubsection{Infinitude of Primes with Full-Period Repeating Decimals}

A prime $p$ is said to have a full-period repeating decimal if $l_p = p-1$, 
which is the maximum possible length for the repeating decimal expansion of $\frac{1}{p}$. 
These primes are also known as full reptend primes or cyclic numbers.

\begin{problem}
    Determine whether there exist infinitely many primes $p$ such that $l_p = p-1$.
\end{problem}

This problem dates back to Gauss, 
who first investigated the properties of repeating decimals and their connection to primitive roots. 
Despite significant attention from mathematicians over the centuries, this question remains unsolved. 
Numerical evidence suggests that approximately 37.35\% of primes have this property, 
but a proof of the infinitude (or finiteness) of such primes has yet to be established.

This problem is connected to other famous unsolved problems in number theory, 
including questions about the distribution of primitive roots and the behavior of the multiplicative group modulo $p$.

\subsubsection{Problems Concerning the Value of $m_p$ in Theorem 2.1}

In Theorem 2.1, 
we introduced the parameter $m_p$ for prime numbers $p$, 
where $c_p = p^{m_p} \cdot k$ (with $p \nmid k$). 
Interestingly, it is known that among primes $p < 10^{12}$, 
only three primes—3, 487, and 56598313—satisfy $m_p > 0$\cite{oeis}. 
Furthermore, for all these primes, $m_p = 1$.

These observations naturally lead to the following two open problems:

\begin{problem}
    Are there infinitely many primes $p$ such that $m_p > 0$?
\end{problem}

\begin{problem}
    Does there exist any prime $p$ such that $m_p \geq 2$?
\end{problem}

Resolving these problems would deepen our understanding of the structure of repeating decimals and potentially provide new insights in number theory.

%\acknow{This work was supported in part by funding from the National Natural Science Foundation of China (Grant No. 62406191). We also thank Ye Tao (Graduate School of Information Science and Technology, The University of Tokyo, tydus@wide.ad.jp) for his contribution to Section 5.2.3 of the Discussion.}

%\showacknow{}

%\begin{comment}
%\end{comment}

%\textbf{Acknowledgements}~~
%\addcontentsline{toc}{section}{Acknowledgements}
%The author wishes to express his sincere thanks to Professor H. Kitada for his encouraging and stimulating discussions with him, 
%and thanks to Professor T. Umeda for his encouraging commucations. 
%The author also wishes to express his sincere thanks to his family for their love. 
%The author would also like to thank Mr. P. Masurel and Mr. J. Le Roux for translating the abstract into French.

%\bibsplit[7]


\begin{thebibliography}{14}%\itemsep 0pt
%\addcontentsline{toc}{section}{References}
    \bibitem{gauss1801disquisitiones}
    Gauss, Carl Friedrich. \textit{Disquisitiones Arithmeticae}. Yale University Press, 1801. (Translated by A. A. Clarke, 1966).
    
    \bibitem{hardywright2008numbers}
    Hardy, G. H., and Wright, E. M. \textit{An Introduction to the Theory of Numbers}. 6th ed., Oxford University Press, 2008.
    
    \bibitem{niven1991numbers}
    Niven, I., Zuckerman, H. S., and Montgomery, H. L. \textit{An Introduction to the Theory of Numbers}. 5th ed., Wiley, 1991.
    
    \bibitem{dickson1952history}
    Dickson, L. E. \textit{History of the Theory of Numbers, Vol. 1: Divisibility and Primality}. Dover Publications, 1952.
    
    \bibitem{koblitz1994cryptography}
    Koblitz, N. \textit{A Course in Number Theory and Cryptography}. 2nd ed., Springer, 1994.

    \bibitem{knuth1997art}
    Knuth, D. E. \textit{The Art of Computer Programming, Volume 2: Seminumerical Algorithms}. 3rd ed., Addison-Wesley, 1997.
    
    \bibitem{crandall2005prime}
    Crandall, R., and Pomerance, C. \textit{Prime Numbers: A Computational Perspective}. 2nd ed., Springer, 2005.
    
    \bibitem{lenstra1993development}
    Lenstra, A. K., and Lenstra, H. W., Jr. (Eds.). \textit{The Development of the Number Field Sieve}. Lecture Notes in Mathematics, Vol. 1554, Springer, 1993.
    
    \bibitem{shoup2009computational}
    Shoup, V. \textit{A Computational Introduction to Number Theory and Algebra}. 2nd ed., Cambridge University Press, 2009.
    
    \bibitem{bach1996algorithmic}
    Bach, E., and Shallit, J. \textit{Algorithmic Number Theory, Vol. 1: Efficient Algorithms}. MIT Press, 1996.
    
    \bibitem{brent1999factorization}
    Brent, R. P. \textit{Factorization of the Tenth Fermat Number}. Mathematics of Computation, 68(225), 429-451, 1999.
    
    \bibitem{wagon1991mathematica}
    Wagon, S. \textit{Mathematica in Action}. W. H. Freeman, 1991.
    
    \bibitem{odlyzko1995future}
    Odlyzko, A. M. \textit{The Future of Integer Factorization}. CryptoBytes, 1(2), 5-12, 1995.

    \bibitem{oeis} 
    Richter, H. \textit{Primes p such that the repeating decimal of 1/p has a repetend containing the digit 0}. The On-Line Encyclopedia of Integer Sequences, 1999. https://oeis.org/A045616.

\end{thebibliography}
\end{document}